\documentclass[a4paper, 12pt]{article}
\usepackage[english]{babel}
\usepackage{xcolor}
\usepackage{xypic}
\usepackage[utf8]{inputenc}
\usepackage[T1]{fontenc}
\usepackage{amsmath}
\usepackage{amsthm}
\usepackage{mathtools}
\usepackage{amssymb}
\usepackage{tikz-cd}
\usepackage{bbm}
\usepackage{tensor}
\usepackage{amsfonts}
\usepackage{thmtools,xcolor}
\usepackage{hyperref}
\usepackage{cleveref}
\usepackage{mathrsfs}
\usepackage{enumitem}
\usepackage{subfiles}
\usepackage{imakeidx}
\usepackage{microtype}
\usepackage[toc,page]{appendix}
\usepackage{crossreftools}
\usepackage{graphicx}
\usepackage{numprint}
\usepackage{setspace}
\usepackage{accents}
\usepackage{emptypage}
\allowdisplaybreaks[1]

\usepackage[a4paper]{geometry}
\theoremstyle{plain}
\newtheorem{theorem}{Theorem}[section]
\newtheorem{cor}[theorem]{Corollary}
\newtheorem{lem}[theorem]{Lemma}
\newtheorem{prop}[theorem]{Proposition}

\theoremstyle{definition}

\newtheorem{dfn}[theorem]{Definition}
\newtheorem{rem}[theorem]{Remark}
\newtheorem{rems}[theorem]{Remarks}

\theoremstyle{remark}

\newcommand{\RNum}[1]{\uppercase\expandafter{\romannumeral #1\relax}}

\makeatletter
\providecommand*{\twoheadrightarrowfill@}{%
  \arrowfill@\relbar\relbar\twoheadrightarrow
}
\providecommand*{\twoheadleftarrowfill@}{%
  \arrowfill@\twoheadleftarrow\relbar\relbar
}
\providecommand*{\xtwoheadrightarrow}[2][]{%
  \ext@arrow 0579\twoheadrightarrowfill@{#1}{#2}%
}
\providecommand*{\xtwoheadleftarrow}[2][]{%
  \ext@arrow 5097\twoheadleftarrowfill@{#1}{#2}%
}
\makeatother

\makeatletter
\def\ind{\@ifnextchar[{\@with}{\@without}}
\def\@with[#1]#2{\mathrm{Ind}(#1,#2)}
\def\@without#1{\mathrm{Ind}(#1)}
\makeatother

\newcounter{para}[section]

\newcommand\KK[0]{K\! K}

\DeclareMathOperator\End{End}

\DeclareMathOperator\ch{Ch}

\DeclareMathOperator\cone{Cone}

\DeclareMathOperator\even{even}
\DeclareMathOperator\odd{odd}

\DeclareMathOperator\id{Id}

\DeclareMathOperator\cs{CS}

\newcommand{\R}{\mathbb{R}}

\newcommand{\C}{\mathbb{C}}
\newcommand{\Z}{\mathbb{Z}}

\newcommand{\w}{\omega}

\begin{document}
\title{Chern-Simons invariants in $\KK$-theory}
\author{Omar Mohsen}
\date{}
\maketitle
\abstract{For a unitary representation $\phi$ of the fundamental group of a compact smooth manifold, Atiyah, Patodi, Singer defined the so called $\alpha$-invariant of $\phi$ using the Chern-Simons invariants. In this article using traces on $C^*$-algebras, we give an intrinsic definition of an element in $\KK$ with real coefficients theory whose pullback by the representation $\phi$ is the $\alpha$-invariant.}
\section*{Introduction}
In \cite{MR0353327}, Chern and Simons defined differential forms associated to a flat vector bundle over a compact smooth manifold. Those invariants naturally live in De Rham cohomology.

Atiyah, Patodi, and Singer \cite{MR0397797,MR0397798,MR0397799} in their celebrated articles highlighted the connection between the Chern-Simons invariants and index theory. They transported the Chern-Simons invariants to $K$-theory. To this end, they defined the $K$-theory with coefficients in $\C/\Z$, and then using Atiyah-Hirzebruch theorem on the bijectivity of the Chern character they transported the Chern-Simons invariants to $K$-theory. The resulting element is the so-called $\alpha$-invariant of a flat vector bundle or equivalently of the holonomy representation of the fundamental group. The $\alpha$-invariant lives in the $K$-theory of the underlying manifold with coefficients in $\C/\Z$. The pairing (Kasparov product) of the $\alpha$-invariant with the class of a Dirac operator $[D]\in \KK^1(M,\C)$ gives the reduced $\eta$-invariant as proved in Atiyah, Patodi, and Singer \cite{MR0397797,MR0397798,MR0397799}.

 If $V$ is a flat vector bundle associated to a representation of the fundamental group of a compact manifold $M$, then the Atiyah-Hirzebruch theorem implies that the element $[V]-[\C^{\dim(V)}]$ in $K^0(M)$ is torsion. A property of the $\alpha$-invariant is that its boundary under Bockstein homomorphism is equal to $[V]-[\C^{\dim(V)}]$. 

Closely related, and in a sense more primitive invariants are the relative Chern-Simons invariants and the relative $\alpha$ invariant which are defined respectively in the De Rham cohomology with coefficients in $\C$ and in the $K$-theory with coefficients in $\C.$ These invariants are defined for flat vector bundles which are equipped with a trivialisation. The relation between the two is that when one takes the relative invariant modulo $\Z$, then the choice of a trivialisation disappears, and the relative invariant becomes the usual invariant. 

When the holonomy representation is unitary, all the different invariants stated above become either in $\R$ or $\R/\Z$. In this article, we restrict ourselves to the relative $\alpha$-invariant of trivialised unitary flat vector bundles.

It was suggested in \cite{MR0397799} that the $\alpha$-invariant should have an intrinsic $K$-theoretical definition that uses the theory of semi-finite Von Neuman algebras. This motivated research in this direction by many authors, e.g. \cite{MR3189427,MR2275011,MR1079841,MR1749950}, etc ...

We continue this line of research by constructing a universal classifying element in the $\KK$-theory of the classifying space of trivialised unitary flat vector bundles. An element directly defined in $\KK$-theory without passing through De Rham cohomology might shed some light on the interaction between Chern-Simons invariants and $\KK$-theory.
 
We follow Antonini, Azzali, Skandalis \cite{MR3189427} definition of $\KK$-theory with real coefficients. By using their work on the $\alpha$-inavraints \cite{MR3189427}, we construct an element in the equivariant $\KK$ theory of the groupoid $U_n\rtimes U_n^\delta$ defined by Kasparov \cite{MR918241} and Le Gall \cite{MR1686846} which when pulled back by the classifying map (seen as a generalised homomorphism in the sense of Hilsum-Skandalis \cite{MR925720}) of a trivialised unitary flat vector bundle $f:M\to U_n\rtimes U_n^\delta$ gives the relative $\alpha$-invariant.

The organisation of this article is as follows.
\smallskip

In  \Cref{sect Chern-Weil}, we recall Chern-Weil theory for tracial $C^*$-algebras. This section follows closely the presentation in the article by Fomenko and Mishchenko \cite{MR548506}.
\smallskip

In \Cref{chapter KK theory}, we recall the definition of real $\KK$-theory given by Antonini, Azzali, and Skandalis \cite{MR3189427}. We refer the reader to \cite{MR1686846,MR1325694,MR1121156,MR1656031} for more details on $\KK$-theory.
\smallskip

In \Cref{sect k theory dfn}, the definition of the Chern Simons invariants in $\KK$-theory is given.
\smallskip

In \Cref{sect global kk theory element}, we construct a more primitive element in equivariant $\KK$-theory, that is done for any compact group.
\section*{Acknowledgements}
I wish to thank my PhD advisor, G. Skandalis for his precious support, suggestions and remarks that were essential to this article. I would also like to thank the referee for his careful reading and very helpful suggestions. This work was supported by grants from Région Ile-de-France.
\section{Chern-Weil theory}\label{sect Chern-Weil}
In this section, we recall Chern-Weil theory for tracial $C^*$-algebras. This section follows closely the articles by Fomenko and Mishchenko \cite{MR548506} and by Simons and Sullivan \cite{MR2521641}. The material presented in this section form the basis of differential cohomology. Differential cohomology was defined by Cheeger and Simons \cite{MR827262}, and since it has been extensively studied, e.g. \cite{MR1076525,MR3873095,MR2443109,MR3529091,MR3462099,MR3289847,MR3897590,MR3852464,MR2231056,MR1724894,MR2664467,MR2357997,MR2129894}.

\bigskip

Let $M$ be a connected smooth manifold, $A$ a unital $C^*$-algebra, $P$ a finitely generated (f.g) projective right $A$-module. A \textit{smooth $A$-vector bundle $V$ with fiber $P$} is a smooth $1$-cocycle on $M$ with coefficients in the group of $A$-linear automorphisms $GL(P)$. We will denote by $\Gamma^\infty(V)$ the right $A$-module of smooth sections of $V$, by $\Omega^\cdot(M,V):=\Omega(M)\hat{\otimes}_{C^\infty(M)} \Gamma^\infty(V)$ the space of differential forms on $M$ with values in $V$. This is a graded module over $\Omega^\cdot(M,A)$.

\begin{dfn}An $A$-connection on $V$ is a $\C$-linear map $\nabla:\Gamma^\infty(V)\to \Omega^1(M,V)$ which satisfies Leibniz rule $$\nabla(sf)=\nabla(s)f+s\otimes df,\quad \forall s\in\Gamma^\infty(V),f\in C^\infty(M,A).$$
\end{dfn}
Like in the classical theory, a connection $\nabla$ extends to a $\C$-linear map $\nabla:\Omega^\cdot(M,V)\to \Omega^{\cdot}(M,V)$ satisfying Leibniz rule. Locally, if $U$ is an open set such that $V\simeq U\times P$, then an $A$-connection is locally a map $\nabla=d+L$, where $L\in \Omega^1(U,\End_A(P)).$ It follows that $\nabla^2$ is the left action by $dL+L^2$. Furthermore connections always exist by a partition of unity argument.

 Let $\tau:A\to \C$ be a finite trace such that $\tau(1)=1$. The trace $\tau$ extends to $M_n(A)$, by the formula $\tau(M)=\sum\tau(M_{i,i})$. The trace extends as well to $\End_A(P)$ for an $A$-projective module $P$ by using a complementary module $Q$, as follows $$\End_A(P)\subseteq \End_A(P\oplus Q)\simeq M_n(A)\xrightarrow\tau \C.$$ It is straightforward to verify that this extension doesn't depend on the choice of $Q$. It follows that if $\nabla$ is an $A$-connection, then for $k\geq 1$, the forms $\tau((dL+L^2)^k)$ glue together to form a complex valued $2k$-form on $M$, that will be denoted by $\tau(\nabla^{2k})$.

\bigskip

Let $V$ be an $A$-vector bundle, and $\nabla$ an $A$-connection on $V$. The Chern character is defined by the formula $$\ch_\tau(V,\nabla):=\exp(\frac{1}{2\pi i}\nabla^2)=\sum_{k=0}^\infty \frac{1}{k!(2\pi i)^k}\tau(\nabla^{2k})\in \Omega^{\even}(M).$$
\begin{rem}
The normalisation constants in the definition of Chern character and the Chern-Simons forms are not uniform across the literature. Some authors don't divide by powers of $\frac{1}{2\pi i}$. Authors divide by powers of $\frac{1}{2\pi i}$ for the Chern character to be a rational map.
\end{rem}
Let $\nabla_t$ be a $C^1$-path of $A$-connections. Locally, if $\nabla_t=d+L_t$, then $\dot{\nabla}_t=\dot{L}_t$ is hence well defined. The forms $\tau(\dot{L}_t\wedge(dL_t+L_t^2)^k)$ glue together to form a complex valued $2k+1$-form on $M$, that will be denoted by $\tau(\dot{\nabla}\wedge\nabla^{2k})$.

The so-called Chern-Simons forms are defined by the formula
\begin{equation}\label{dfn chernsimons forms}
\cs_\tau(V,\nabla_t):=\int_0^1\sum_{k=0}^\infty\frac{1}{k!(2\pi i)^k}\tau(\dot{\nabla_t}\wedge\nabla_t^{2k})\in \Omega^{\odd}(M) .
\end{equation}
By a direct local computation, one deduces that  \begin{equation}\label{chernsimons eqn}
d\cs_\tau(V,\nabla_t)=\ch_\tau(V,\nabla_1)-\ch_\tau(V,\nabla_0).
\end{equation}
Hence the following holds
\begin{prop}\label{chern form are closed}\label{prop of chern simons even}The Chern character is a closed differential form whose class is independent of the choice of the connection $\nabla$.
\end{prop}
\begin{proof}
To see that the form $\ch(V,\nabla)$ is closed one compares it locally to the trivial connection using \Cref{chernsimons eqn}.
\end{proof}
If $\tilde{\nabla}_t$ is another $C^1$-path of connections with same endpoint as $\nabla_t$, then in \cite{MR2521641}, complex valued differential forms $\cs_\tau(V,\tilde{\nabla}_t,\nabla_t)$ are defined such that $$d\cs_\tau(V,\tilde{\nabla}_t,\nabla_t)=\cs_\tau(V,\nabla_t)-\cs_\tau(V,\tilde{\nabla}_t).$$ It follows that modulo $d\Omega^{\even}(M)$, the forms $\cs_\tau(V,\nabla_t)$ only depend on the endpoint $\nabla_{0}$ and $\nabla_{1}$. Hence the notation $\cs_\tau(V,\nabla_1,\nabla_0)\in \Omega^{\odd}(M)/d\Omega^{\even}(M)$ is justified.
Given two connections $\nabla_0,\nabla_1$, then there is a preferred path $t\nabla_0+(1-t)\nabla_1$. One sees that in this case \Cref{dfn chernsimons forms} becomes 
\begin{equation}\label{chern simons formula using Ak}
\cs_\tau(\nabla_1,\nabla_0)=\sum_{k=0}^\infty \frac{(-1)^kk!}{(2\pi i)^{k+1}(2k+1)!}\tau((\nabla_1-\nabla_0)^{2k+1})
\end{equation}

Let us recall the tensor product of bundles: if $A$ and $B$ are $C^*$-algebras, $P$ a f.g projective $A$-module, $Q$ a f.g projective $B$-module, then $P\otimes_{max} Q$ is a f.g projective $A\otimes_{\max}B$-module. It follows that if $V$ is a $A$-bundle and $W$ is a $B$-module then the maximal tensor product $V\otimes_{max} W$ is a well defined smooth $A\otimes_{max}B$-vector bundle. Same holds for minimal tensor product.

\begin{prop}[\cite{MR2521641}]\label{cherncharacterschernSimonsIdentites} Let $V,W$ be $A$-vector bundles and $\nabla^{0}_V,\nabla^1_V,\nabla^2_V,\nabla^0_W,\nabla^1_W$ be $A$-connections on the indicated bundles. Then, we have  \begin{enumerate}

\item $\ch_\tau(V\oplus W,\nabla^0_V\oplus\nabla^0_W)=\ch_\tau(V,\nabla^0_V)+\ch_\tau(W,\nabla^0_W)$
\item $\ch_\tau(V\otimes W,\nabla^0 _V\otimes\nabla^0_W)=\ch_\tau(V,\nabla^0_V)\wedge\ch_\tau(W,\nabla^0_W)$
\item $\cs_\tau(\nabla^0_V,\nabla^1_V)+\cs_\tau(\nabla^1_V,\nabla^2_V)=\cs_\tau(\nabla^0_V,\nabla^2_V)$
\item $\cs_\tau(\nabla^0_V\oplus\nabla^0_W,\nabla^1_V\oplus\nabla^1_W)=\cs_\tau(\nabla^0_V,\nabla^1_V)+\cs_\tau(\nabla^0_W,\nabla^1_W)$
\item $$\cs_\tau(\nabla^0_V\otimes\nabla_W,\nabla^1_V\otimes\nabla_W)=\ch_\tau(\nabla_W)\cs_\tau(\nabla^0_V,\nabla^1_V),$$  where the product $\ch_\tau(\nabla_W)\cs_\tau(\nabla^0_V,\nabla^1_V)$ is well defined module exact forms because $\ch_\tau(\nabla_W)$ is closed.
\item\begin{align*}
\cs_\tau(\nabla^0_V\otimes\nabla^0_W,\nabla^1_V\otimes\nabla^1_W)=\ch_\tau(\nabla^0_V)\cs_\tau(\nabla^0_W,\nabla^1_W)\\+\ch_\tau(\nabla^1_W)\cs_\tau(\nabla^0_V,\nabla^1_V)
\end{align*}
\end{enumerate}
\end{prop}

The odd Chern-character (cf. \cite{MR1231957}) is defined as follows. Let $u$ be an $A$-linear automorphism of an $A$-vector bundle $V$, and $\nabla$ an $A$-connection on $V$, then the odd Chern character is defined by the formula
\begin{equation}\label{chern char odd}
\ch_\tau(u,\nabla)=\sum_{k=0}^\infty\frac{(-1)^kk!}{(2k+1)!(2\pi i)^{k+1}}\tau((u^{-1}\nabla u-u)^{2k+1}),
\end{equation}
where $u^{-1}\nabla u$ is the $A$-connection on $V$ defined on $V$ by the formula \begin{align*}
\Omega^*(M,V)\to \Omega^{*+1}(M,V),\quad s\to u^{-1}\nabla(us)
\end{align*}
\begin{cor}\label{chern simons=oddchern char}Let $\nabla$ be an $A$-connection on an $A$-vector bundle $V$ and $T:V\to V$ an $A$-linear automorphism, then we have \begin{equation}
\cs_\tau(T^{-1}\nabla T,\nabla)=\ch_\tau(T,\nabla)
\end{equation}
\end{cor}
\begin{proof}
This follows from \Cref{chern simons formula using Ak} and \Cref{chern char odd}.
\end{proof}
\begin{prop}[\cite{MR2521641}]
The form $\ch_\tau(u)$ is closed, and its class depends only on the homotopy class of $u$.
\end{prop}

\begin{theorem}[Atiyah-Hirzebruch, see \cite{MR0370578}]\label{Atiyah Hirz}Let $M$ be a compact smooth manifold, then the Chern character $\ch:K^*(M)\otimes \C\to H^{*}(M,\C)$ is a ring isomorphism.
\end{theorem}

Every finitely generated projective $A$-module can be given the structure of a $C^*$-module over $A$ (see \cite{MR1325694} for more details). We will call such a structure a $C^*$-metric. A $C^*$-metric on an $A$-bundle $V$ is a smooth family of $C^*$-metrics on each fiber. If $g$ is a $C^*$-metric, $\nabla$ an $A$-connection on $V$, then $\nabla^*$ is an $A$-connection defined by the equation for $X$ a vector field, $s,s'\in \Gamma^\infty(V)$, $$Xg(s,s')=g(\nabla_Xs,s')+g(s,\nabla^*_Xs').$$
A connection $\nabla$ is called unitary if $\nabla^*=\nabla.$ 

The following equation follows from \cref{chern simons formula using Ak}, \begin{equation}\label{csadjoint}
\overline{\cs(\nabla_0,\nabla_1)}=\cs(\nabla_0^*,\nabla_1^*).
\end{equation}

\bigskip

Let $\Gamma$ be the fundamental group of $M$, $\tilde{M}$ its universal cover, $\phi:\Gamma\to GL(P)$ a representation. One can define an $A$-vector bundle by $\tilde{M}\times_\Gamma P$.
\begin{dfn}\label{dfn flat bundle C} A flat structure on an $A$-vector bundle on $M$ is the choice of an $A$-vector bundle isomorphism to $\tilde{M}\times_\Gamma P$ for some representation $\phi:\Gamma\to GL(P)$, and for some finitely generated projective $A$-module $P$. Furthermore if $P$ is endowed with the structure of a $C^*$-module such that $\phi$ is a unitary representation, then we say that $V$ is unitary flat,

A flat vector bundle is a vector bundle equipped with a flat structure.
\end{dfn}
\begin{prop}\label{flat vector bundle eqv}A flat structure on an $A$-vector bundle can be equivalently given by the choice of a flat $A$-connection, that is an $A$-connection $\nabla$ such that $\nabla^2=0$. Furthermore, the bundle is unitary flat if and only if the connection is unitary with respect to a given $C^*$-metric on the vector bundle. 

The representation $\phi:\Gamma\to GL(P)$ associated to $\nabla$ is called the \textit{holonomy representation} of $\nabla$.
\end{prop}
\begin{dfn}A trivial $A$-connection is a flat $A$-connection $\nabla$ whose holonomy is trivial.
\end{dfn}
\begin{rem}It is clear from the definition that giving a trivial connection on a bundle is the same as giving a trivialization of the bundle.\end{rem}
Let $\nabla_0,\nabla_1$ be flat $A$-connections on an $A$-vector bundle $V$. It follows immediately from the definition of the Chern character that $\ch_\tau(\nabla_0)=\ch_\tau(\nabla_1)=\tau(Id_P)$, where $Id_P\in \End_A(P)$ is the identity morphism. It follows from \Cref{chernsimons eqn} that $\cs_\tau(\nabla_1,\nabla_0)$ gives a cohomology class in $H^{\odd}(M,\C)$.

\begin{dfn}The $\alpha$-invariant of $(V,\nabla_1,\nabla_0)$ is defined as $$\alpha_{V,\nabla_1,\nabla_0}=\ch^{-1}(\cs_\tau(\nabla_1,\nabla_0))\in K^1(M,\C).$$
\end{dfn}

From now on we restrict our selves to the case of unitary representations. In this case the imaginary part of the $\alpha$ invariant is zero as can be immediately seen from \Cref{chern simons formula using Ak}. In general this holds for connections that are selfadjoint with respect to a nondegenerate sesquilinear forms introduced  by Hilsum and Skandalis \cite{MR1142484} . 

A nondegenerate sesquilinear form is a $\C$-bilinear form $Q:P\times P\to A$ such that $Q(p,q)=Q(q,p)^*$, $Q(p,qa)=Q(p,q)a$, and that there exists a bijective $A$-linear operator $T:P\to P$ such that $Q(\cdot ,T\cdot)$ is a $C^*$-metric. It is proved in \cite{MR1142484} that $T$ can be chosen so that $T^2=1$.
\begin{prop}Let $V$ be an $A$ vector bundle with fiber $P$, $\nabla$ a flat connection whose holonomy is $\phi$ and $\nabla_{triv}$ a trivial connection on $V$. If there exists a non degenerate sesquilinear form $Q$ on $V$ such that $\phi(\Gamma)\subseteq U(Q)$, then the imaginary part of $\alpha_{\nabla,\nabla_{triv}}$ vanishes. Here $U(Q)$ denotes the group of isometries of $Q$.
\end{prop}
\begin{proof}
Let $T$ an operator acting on $P$, and $g$ a $C^*$metric such that $Q(\cdot,\cdot)=g(\cdot,T\cdot)$ and $T^2=1$. We will denote by $\nabla^*$ the adjoint of $\nabla$ with respect to $g$. Let $s,s'\in\Gamma(V)$ be two sections and $X\in \Gamma(TM)$ a vector field. Then \begin{align*}
Q(\nabla_Xs,s')=g(\nabla_Xs,Ts')&=X\cdot g(s,Ts')-g(s,\nabla^*_Xs')\\
&=X\cdot Q(s,s')-Q(s,T\nabla^*_X Ts')
\end{align*}
It follows that $\nabla=T\nabla^* T$ and hence the form $\w=\nabla-\nabla^*$ anticommutes with $T$. Hence \begin{equation}\label{eqn used proof chern simons imag trivial}
\tau(\omega^k)=\tau(\w^k T^2)=\tau((-1)^kT\w^k T)=(-1)^k\tau(\w^k).
\end{equation} By \cref{csadjoint}, one has \begin{align*}
\overline{\cs(\nabla,\nabla_{triv})}=\cs(\nabla^*,\nabla_{triv}^*)=\cs(\nabla^*,\nabla_{triv}).
\end{align*} It follows that the imaginary part of Chern-Simons form $\cs(\nabla,\nabla_{triv})$ is equal to $\frac{1}{2i}\cs(\nabla,\nabla^*)$. The result then follows from \Cref{chern simons formula using Ak} and \Cref{eqn used proof chern simons imag trivial}.
\end{proof}
\section{$\KK$-theory with real coefficients}\label{chapter KK theory}Let $G$ be a Lie groupoid. We will use equivariant $\KK$-théorie as defined in \cite{MR918241,MR1686846}. We recall that if $A$, $B$ are separable $C^*$-algebras, $G$ a Hausdroff Lie groupoid, the group $KK_G^*(A,B)$ is an abelian group defined by cycles $[E,\pi,F]$, where $E$ is a $G$-equivariant $B$-module, $\pi$ is a $G$-equivariant action of $A$ on $E$, $F\in \mathcal{L(E)}$ a bounded operator with some compactness conditions.

\begin{rem}\label{final rem}It follows from the definition of $KK_G^*(A,B)$, that if $\Gamma$ is a discrete group acting on a compact space $X$, $A$ a $C(X)\rtimes \Gamma$-$C^*$-algebra then $$KK_{X\rtimes \Gamma}^*(C(X),A)=KK_\Gamma^*(\C,A).$$ This holds because a cycle $[E,\pi,F]$ in $KK_{X\rtimes \Gamma}^*(C(X),A)$ is just a cycle in $KK_\Gamma^*(\C,A)$ with an added action of $C(X)$ on $E$ on the left, but this action is already implemented by the right action of $A$ on $E$.
\end{rem}

 The $\KK^*_G(A,B)$ group is usually defined only for separable $C^*$-algebras only. We follow the remarks given by Skandalis \cite{MR775997} in order to define $\KK^*_G(A,B)$ for arbitrary $C^*$-algebras $A$ and $B$ by $$\KK^*_G(A,B):=\varprojlim_D\varinjlim_{C}\KK^*_G(D,C)$$ where the projective and injective limits are over all separable $C^*$-algebras $D$ and $C$ with morphisms $\psi:A\to D$ and $\phi:C\to B$.

When the groupoid $G$ is not second countable (but $G^0$ is always assumed second countable) then we define equivariant $KK$-theory by $$\KK^*_G(A,B):=\varprojlim_{H}\KK^*_H(A,B)$$ where the projective limit is over all second countable Lie subgroupoids.

\begin{dfn}Let $\phi:A\to B$ be a homomorphism. The mapping cone of $\phi$ is the $C^*$-algebra $$\cone(\phi)=\{(a,f)\in A\times C_0([0,1),B):f(0)=\phi(a)\}.$$
\end{dfn}

\begin{prop}\label{Morph to exact seq k theory}
The following sequence is exact
\begin{equation}
0\to SB\to \cone(\phi)\to A\to 0
\end{equation}
Its boundary morphism in $\KK^1(A,SB)=\KK^0(A,B)$ is equal to $\phi_*$. Therefore the following sequence is exact.
\[
\begin{tikzcd}[row sep=2 em,column sep=1em]
K_0(A)\arrow{r}{\phi_*}&K_0(B)\arrow[r] &K_1(\cone(\phi))\arrow{d} \\
   K_0(\cone(\phi))\arrow[u]&K_1(B)\arrow[l]&K_1(A)\arrow{l}{\phi_*}\\
\end{tikzcd} \]
\end{prop}

\begin{dfn}Let $\mathcal{C}$ be the category whose objects are separable unital $C^*$-algebras endowed with a tracial state and whose morphisms are $*$-morphisms preserving the trace.
\end{dfn}
\begin{dfn}[\cite{MR3419768}]Let $G$ be a Lie groupoid, and let $A$ and $B$ be two $G$-$C^*$-algebras. Equivariant $\KK$-theory with real coefficients is defined by $$\KK_{G,\R}^*(A,B):=\varinjlim_{C\in \mathcal{C}}\KK^*_{G}(A,B\otimes C).$$ Here the groupoid $G$ acts on $C$ trivially.

Similarly the equivariant $\KK$-theory with $\R/\Z$ coefficients is defined by $$\KK_{G,\R/\Z}^*(A,B)=\varinjlim_{C\in \mathcal{C}}\KK^*_{G}\left(A,B\otimes \cone(\C\to C)\right).$$
\end{dfn}
This definition is justified by the Kunneth formula and following proposition.
\begin{prop}Let $A$ be a $C^*$-algebra, then the exact sequence in \cref{Morph to exact seq k theory} induces under injective limit the exact sequence \[
\begin{tikzcd}[row sep=2 em,column sep=1em]
\KK^0(\C,A)\arrow{r}{}&\KK^0_{\R}(\C,A)\arrow[r] &\KK^0_{\R/\Z}(\C,A)\arrow{d} \\
   \KK^1_{\R/\Z}(\C,A)\arrow[u]&\KK^1_{\R}(\C,A)\arrow[l]&\KK^1(\C,A)\arrow{l}{}\\
\end{tikzcd} \]
\end{prop}

 \begin{theorem}[Künneth formula, see for example \cite{MR1656031}]\label{Kunneth formula}
Let $A$ be a separable $C^*$-algebra in the bootstrap category and $B$ any $C^*$-algebra then the following sequences are exact

\begin{equation*}
 0\to K_*(A)\otimes K_*(B)\to K_*(A\otimes B)\to \mathrm{Tor}^1_\Z(K_*(A),K_*(B))\to0
\end{equation*}
Where the first map is of degree 0 and the second is of degree $1$.
\end{theorem}
 \begin{prop}[{\cite{MR3189427}}]Let $M$ be a compact smooth manifold, then $$\KK^*_{\R}(\C,C(M))=K^*(M,\R).$$
\end{prop}

Theorems 6.3, 6.4, 7.1, 7.2, 7.3, 7.4, 7.5, 7.6 and propositions 6.2, 6.3, 7.1, 7.2, and corollaire 7.1  in \cite{MR1686846} pass through the direct limit to $\KK_{G,\R}$. In particular Kasparov product exists $$\KK^i_{G,\R}(A,B)\times \KK^j_{G,\R}(B,C)\to \KK^{i+j}_{G,\R}(A,C).$$ Functoriality and Morita equivalence hold that is if $f:G\to G'$ is a generalised morphism, in the sense of \cite{MR925720}, of groupoids then $$f^*:\KK_{G',\R}^*(A,B)\to \KK^*_{G,\R}(f^*A,f^*B)$$ is well defined, moreover if $f$ is a Morita equivalence, then $f^*$ is an isomorphism.

If $\tau:A\to \C$ is a trace on a $C^*$-algebra, then $\tau$ defines naturally a class $[\tau]$ in $\KK^0_{\R}(A,\C)$. 

More generally if $\Gamma$ is a discrete group acting on $A$ which preserves the trace $\tau$, then we will define a class $[\tau]$ in $KK_{\Gamma,\R}^0(A,\C)$. The class $[\tau]=x\otimes y\otimes z$ is defined as the product of the following three elements $$x\in KK^0_{\Gamma}(A,A_1),y\in KK^0_{\Gamma}(A_1,A_2),z\in KK_{\Gamma,\R}(A_2,\C),$$ where $A_1=A\rtimes \Gamma$ equipped with the $\Gamma$-inner action, and $A_2=A\rtimes \Gamma$ with the trivial action.\begin{enumerate}
\item The element $x$ is the class of the inclusion map $i:A\to A_1$ which is $\Gamma$-equivariant.
\item Since the $\Gamma$ action on $A_1$ is inner, $A_1$ is Morita equivalent to $A_2$. The element $y$ is class of this Morita equivalence.
\item Since the trace $\tau$ is $\Gamma$-invariant, the map $$\tau_{\rtimes\Gamma}(\sum_{\gamma}a_{\gamma}\gamma)=\tau(a_{e})$$ is a trace on $A_2$. The element $z$ is then the class of $\tau_{\rtimes \Gamma}$ in $KK^0_{\Gamma,\R}(A_2,\C).$
\end{enumerate}
We will need this construction later on, so let us summarise the previous construction in the following proposition.
\begin{prop}\label{KK theory inclusion- morita then R KK}
Let $A$ be a $C^*$-algebra, , $\Gamma$ a discrete group acting on $A$ which preserves $\tau$. Using previous notations, the trace $\tau$ defines a class $[\tau]=x\otimes y\otimes z$ in $KK^0_{\Gamma,\R}(A,\C)$.
\end{prop}

\section{$\KK$-theoretic definition of $\alpha$-invariants}\label{sect k theory dfn}
In this section, we recall in \Cref{main prop} the description of the Chern-Simons invariants due to Antonini, Azzali and Skandalis. To this end  \Cref{prop chern simons M T V times W} will be the main step. We extend \Cref{main prop} in  \Cref{prop chern simons M T V times W 2} to the noncommutative case.

\bigskip

\begin{theorem}[\cite{MR3189427}]\label{prop chern simons M T V times W}Let $V$ be a $\C$-vector bundle over a manifold $M$, $\nabla_V$ a unitary flat connection, and $\nabla_{triv}$, a trivial connection. There exists a unital $C^*$-algebra $A$ equipped with a trace $\tau$ such that $\tau(1)=1$, and a unitary flat $A$-bundle $(W,\nabla_W)$ whose fiber is equal to $A$, and an isomorphism $T:V\otimes W\to V\otimes W$ such that 
\begin{equation}\label{eqn chern simons T interwing Vtimes W}
T^{-1}(\nabla_{triv} \otimes \nabla_W)T=\nabla_V\otimes \nabla_W.
\end{equation}
\end{theorem}
The following is the key theorem of this article.
\begin{theorem}\label{main prop}Following the notation in \cref{prop chern simons M T V times W}, let $[T]\in KK^1(\C,C(M)\otimes A)$ be the class of a map $T$ which satisfies \cref{eqn chern simons T interwing Vtimes W} and $[\tau]\in KK^0_{\R}(A,\C)$ the element in real $\KK$ theory defined by the trace $\tau$. The Kasparov product $[T]\otimes_A[\tau]\in \KK^1_\R(\C,C(M))=K^1(M,\R)$ is equal to the $\alpha$ invariant $\alpha_{\nabla,\nabla_{triv}}$. In particular it is independent of the choice of $A$ and $W$.
\end{theorem}
\begin{proof}
By \Cref{chern simons=oddchern char}, we have \begin{align*}
\ch_\tau([T])=&[\cs_\tau(T^{-1}(\nabla_{triv}\otimes\nabla_{W})T,\nabla_{triv}\otimes\nabla_W]\\
=&[\cs_\tau(\nabla_V\otimes\nabla_W,\nabla_{triv}\otimes\nabla_W)]\\
=&[\cs_\tau(\nabla_V,\nabla_{triv})][\ch_\tau(\nabla_W)]\\
=&\ch(\alpha_{V,\nabla,\nabla_{triv}})\tau(1)=\ch(\alpha_{V,\nabla,\nabla_{triv}})
\end{align*}
It follows that the class $[T]\otimes_A[\tau]$ in $K^1(M,\R)$ is equal to $\alpha_{\nabla,\nabla_{triv}}$ 
\end{proof}
\begin{rems}
\begin{enumerate}
\item In general, it is impossible to find a commutative algebra $A$ satisfying \Cref{prop chern simons M T V times W}. Because if such an algebra exists, the Chern-Simons invariants become rational by the rationality of the Chern-character on locally compact spaces which doesn't hold in general.
\item The proposition is in general false for non unitary flat connections, because  in the case where it holds, the imaginary part of Chern-Simons invariant is equal to $0$.
\end{enumerate}
\end{rems}
In the next proposition, we extend \cref{prop chern simons M T V times W} to an arbitrary noncommutative $C^*$-algebra other than $M_n(\C)$.
\begin{prop}\label{prop chern simons M T V times W 2}Let $A$ be a unital $C^*$-algebra, $V$ a unitary $A$-flat vector bundle. There exists a unital $C^*$-algebra $B$, a $*$-morphism $i:A\otimes C^*_r\Gamma\to B$ such that if $W$ denotes Mishchenko's universal $C^*_r\Gamma$-bundle, then there exists an isomorphism preserving the flat structure $$T:i_*((M\times A)\otimes W)\to i_*(V\otimes W),$$ where $M\times A$ is the trivial $A$-bundle over $M$ with fiber $A$. Furthermore if $\tau_A$ is a tracial state on $A$, then a tracial state $\tau_B$ on $B$ is naturally defined such that $\tau_B(i(a\otimes \gamma))=\tau_A(a)\delta_{e}(\gamma).$ 
\end{prop}
\begin{proof}
Let $S^1$ be the unit circle, $z\in C(S^1)$ denote the identity map. The free product $A\star_\C C(S^1)$ can be described as the universal unital $C^*$-algebra that is equipped with a $*$-morphism $i:A\to A\star_\C C(S^1)$ and a unitary $z\in  A\star_\C C(S^1)$. In other words, if $B$ is a $C^*$-algebra with a $*$-morphism $j:A\to B$ and a unitary $w\in B$, then there exists a unique $*$-morphism $\phi:A\star_\C C(S^1)\to B$ such that $\phi\circ i=j$ and $\phi(z)=w$. See \cite{MR799593} for more details on free products.

Let $u\in A$ be a unitary, then $u^{-1}z$ is a unitary in $A\star_\C C(S^1)$, hence by the universality of $A\star_\C C(S^1)$, there exists a unique $*$-morphism $$\phi_u:A\star_\C C(S^1)\to A\star_\C C(S^1)$$ such that $\phi_u(a)=a$, and $\phi_u(z)=u^{-1}z$. By uniqueness, one has $\phi_u\circ\phi_v=\phi_{uv}$. Since $\phi_{1}=\id$, it follows that $\phi_u$ is an automorphism for every $u$.

Let $\psi:\Gamma\to U(A)$ be the holonomy representation of $V$. The group $\Gamma$ acts on $A\star_\C C(S^1)$ by the morphisms $\phi_{\psi(\gamma)}$ for $\gamma\in \Gamma.$

 Let $B=(A\star_\C C(S^1))\rtimes_r \Gamma$, $i$ the natural $*$-morphism given by $$i:A\otimes C_r^*\Gamma\to B,\quad i(a\otimes \gamma)=a\gamma.$$ The map $i$ is a $*$-morphism because $\phi_u$ fixes $A$ for any unitary $u$. The unitary $z\in A\star_\C C(S^1)\subseteq B$ satisfies the following equation $$i(1\otimes \gamma)zi(1\otimes \gamma^{-1})=\gamma z\gamma^{-1}=\phi_{\psi(\gamma)}(z)=\psi(\gamma)^{-1}z,$$ which means that $z$ defines an isomorphism $i_*((M\times A)\otimes W) \to i_*(V\otimes W)$ which preserves the flat structure.
 
 \bigskip

Let $\tau_A:A\to\C$ be a tracial state. We will denote by $\ker(\tau_A)\subseteq A$ the kernel of $\tau_A$. By \cite[{section 1}]{MR799593}, the algebra $A\star_\C C(S^1)$ admits a finite trace $\tau=\tau_A\star \int_{S^1}$. This trace is defined as the unique trace satisfying the following properties  \begin{enumerate}
\item If $a\in A$, then $\tau(a)=\tau_A(a)$
\item If $k\in \Z^*$, then $\tau(z^k)=0$
\item If $k\geq 1$, $a_1,\dots a_k\in \ker(\tau_A)\subseteq A$ and $l_1,\dots l_k\in \Z^*$, then $$\tau(a_1z^{l_1}a_2z^{l_2}\cdots a_kz^{l_k})=0.$$
\end{enumerate}
\begin{lem}\label{tau start int S1 is invariant}
For every unitary $u\in A$, the automorphism $\phi_u$ preserves the trace $\tau$.
\end{lem}
\begin{proof}
Let $B_+\subseteq A\star_\C C(S^1)$ be the linear span of elements of the form $z^{l_0}a_{1}z^{l_1}a_{2}\dots a_kz^{l_k}$ for $k\geq 0$, $l_i>0,$ $a_i\in A$. Let $B_+^*$ be the adjoints of elements in $B_+$. Every element in $B_+$ can be written as a finite sum of elements of the form $z^{l_0}a_{1}z^{l_1}a_{2}\dots a_kz^{l_k}$ with $l_i>0$ and $a_i\in \ker(\tau_A)\subseteq A$. Hence $\tau_{|B_+}=0$, furthermore for $k\geq 1$, $b_i\in B_{+}\coprod B_+^*$ and $a_i\in \ker(\tau_A)\subseteq A$, one has $$\tau(a_1b_1a_2\dots a_kb_k)=0.$$

Furthermore it follows from the the identity $\tau(ab)=\tau((a-\tau(a))b)+\tau(a)\tau(b)$ that $\tau(ab)=0$ for $a\in A$, $b\in B_+.$

We check that the properties of $\tau$, are also verified by $\tau\circ\phi_u$. The result then follows from uniqueness of the trace. 

\begin{enumerate}
\item If $a\in A$, then $\tau(\phi_u(a))=\tau(a)=\tau_A(a)$.
\item  If $k>0$, then $\phi_u(z^k)=(u^{-1}z)^k$. Hence $\phi_u(z^k)=u^{-1}x\in AB_+$. Therefore $\tau(\phi_u(z^k))=0$. By taking the adjoint it follows that $\tau(\phi_u(z^{-k}))=0$.
\item Any element $a_1z^{l_1}a_2z^{l_2}\cdots a_kz^{l_k}$ for $l_i\in \Z^*$ and $a_i\in \ker(\tau_A)\subseteq A$ can be written as $\alpha_1 x_1\alpha_2 x_2\dots \alpha_s x_s$ for some $s\geq 1$, $\alpha_i\in \ker(\tau_A)\subseteq A$ and $x_i\in B_+\coprod B_+^*$ such that $x_i$ alternatively belong to $B_+$ and $B_+^*$.

Suppose that $x_1\in B_+$. Since $\phi_u(B_+)=u^{-1}B_+$ and $\phi_u(B_+^*)=B_+^*u$, by writing $\phi_u(x_i)=u^{-1}y_i$ or $y_iu$, it follows that \begin{equation*}
\phi_u(\alpha_1 x_1\alpha_2 x_2\dots \alpha_s x_s)=\begin{cases}\alpha_1u^{-1}y_1\alpha_2y_2(u\alpha_3 u^{-1})y_3\dots y_su,\quad \text{if},\; u_s\in B_+^*\\\alpha_1u^{-1}y_1\alpha_2y_2(u\alpha_3 u^{-1})y_3\dots y_s,\quad \text{if},\; u_s\in B_+\end{cases}
\end{equation*}
It then follows that $\tau(\phi_u(\alpha_1 x_1\alpha_2 x_2\dots \alpha_s x_s)$ is equal to; \begin{align*}
\tau(\alpha_1u^{-1}y_1\alpha_2y_2(u\alpha_3 u^{-1})y_3\dots y_su)=\tau((u\alpha_1u^{-1})y_1\alpha_2y_2(u\alpha_3 u^{-1})y_3\dots y_s)\\=0
\end{align*}in the first case  and in the second we have \begin{align*}
\tau(\alpha_1u^{-1}y_1\alpha_2y_2(u\alpha_3 u^{-1})y_3\dots y_s)=\tau(\alpha_2y_2(u\alpha_3 u^{-1})y_3\dots \left(y_s\alpha_1 u^{-1}y_1\right))=0,
\end{align*}
where we used that $y_s\alpha_1 u^{-1}y_1\in B_+.$

Hence $\tau(a_1z^{l_1}a_2z^{l_2}\cdots a_kz^{l_k})=0$ in the case where $x_1\in B_+$. The case where $x_1\in B_+^*$ is handled similarly.\qedhere
\end{enumerate}
\end{proof}
It follows from \Cref{tau start int S1 is invariant} that $(A\star C(S^1))\rtimes \Gamma$, admits a tracial state $\tau\rtimes \tau_e$ defined by $\tau\rtimes\tau_e(\sum b_i\gamma_i)=\tau (b_{e})$, where $b_i\in A\star C(S^1).$
\end{proof}
\section{A morphism in $\KK$-theory with real coefficients}\label{sect global kk theory element}
In this section, we give the definition of a primitive element in the equivariant $\KK$-theory of the classifying space of trivialised unitary flat vector bundles.

\bigskip

Let $G$ be a compact Lie group. To avoid confusion, we will denote by $G$ the Lie group seen as a Lie group, $G^\delta$ the group $G$ with the discrete topology, $\mathbb{G}$ the space $G$ seen as a compact space.
\begin{dfn}\label{dfn morpho secondary KK compact}
The following morphism defined below is denoted $\Psi_G$
\begin{equation}\label{KK morph G to GR compact}
\Psi_G:\KK^*(\C,C(\mathbb{G}))\to \KK^{*}_{G^\delta,\R}(\C,C(\mathbb{G}))=\KK^{*}_{\mathbb{G}\rtimes G^\delta,\R}(C(\mathbb{G}),C(\mathbb{G})),
\end{equation}
where $G^\delta$ acts by right translation on $\mathbb{G}.$ The morphism is the successive composition of the following morphisms \begin{enumerate}
 \item Let $G$ act on $\mathbb{G}\times \mathbb{G}$ by the right diagonal action. The space $(\mathbb{G}\times \mathbb{G})/G$ is identified with $\mathbb{G}$ by using the map $(x,y)\to yx^{-1}$. It follows that we have a Morita equivalence from the $C^*$-algebra $C(\mathbb{G})$ to the $C^*$-algebra  $C(\mathbb{G}\times \mathbb{G})\rtimes G$. By Green-Julg theorem \cite{MR625361}, we obtain an isomorphism $$\KK^*(\C,C(\mathbb{G}))\to \KK^*_G(\C,C(\mathbb{G})\otimes C(\mathbb{G})).$$
 \item The forgetful map and then the inclusion map of $\KK$-theory inside $KK_\R$-theory $$\KK^*_{G}(\C,C(\mathbb{G}\times \mathbb{G}))\to \KK^*_{G^\delta,\R}(\C,C(\mathbb{G})\otimes C(\mathbb{G})).$$
\item By \Cref{KK theory inclusion- morita then R KK}, the Haar measure defines an element in $KK^0_{G^\delta,\R}(C(\mathbb{G}),\C)$. One takes the Kasparov product with this element (on the second copy on $C(\mathbb{G})$) to obtain a morphism $$\KK^*_{G^\delta}(\C,C(\mathbb{G})\otimes C(\mathbb{G}))\to KK^*_{G^\delta,\R}(\C,C(\mathbb{G}))$$
\item The equality $\KK^{*}_{G^\delta,\R}(\C,C(\mathbb{G}))=\KK^{*}_{\mathbb{G}\rtimes G^\delta,\R}(C(\mathbb{G}),C(\mathbb{G}))$ follows from \cref{final rem}.
 \end{enumerate}
\end{dfn}
\begin{rem}
 The composition of the morphism (\ref{KK morph G to GR compact}) with the forgetful morphism $\KK^{*}_{G^\delta,\R}(\C,C(\mathbb{G}))\to \KK^{*}_{\R}(\C,C(\mathbb{G}))$ is the inclusion morphism of $\Z$ in $\R$.
\end{rem}

\bigskip

Following our convention, $U_n$ denotes the Lie group, $\mathbb{U}_n$ the space, $U_n^\delta$ the discrete group. Let $M$ be a compact manifold, $\Gamma=\pi_1(M)$. The groupoids $M$ and $\tilde{M}\rtimes\Gamma$ are Morita equivalent. Since a flat vector bundle $V$ whose holonomy representation is $\phi:\Gamma\to U_n^\delta$ defines a continuous functor of groupoids $\tilde{M}\rtimes\Gamma\to U_n^\delta$ by sending $(x,\gamma)$ to $\phi(\gamma)$. It follows that when this morphism is composed with the Morita equivalence of $M$ and $\tilde{M}\rtimes\Gamma$, a flat vector bundle defines a generalised morphism\footnote{See \cite{MR925720} for the definition of a generalised morphism.} denoted by $f_\phi:M\to U_n^\delta.$

One sees easily that giving a trivialisation of a $\tilde{M}\times_\phi\C^n$ is the same thing as a map $\beta:\tilde{M}\to U_n$ such that $\beta(x\gamma)=\beta(x)\phi(\gamma)$ for every $x\in \tilde{M}$ and $\gamma\in\Gamma$. In particular if $V$ is trivialised flat vector bundle, then the following morphism \begin{align*}
 \tilde{M}\rtimes \Gamma\to U_n\rtimes U_n^\delta,\quad (x,\gamma)\to (\beta(x),\phi(\gamma)).
\end{align*} is a morphism of Lie groupoids. By composing with the Morita equivalence of $M$ and $\tilde{M}\rtimes\Gamma$, one obtains a generalised morphism denoted by $f_V:M\to U_n\rtimes U^\delta$.
\begin{theorem}
Let $G=U_n$, and $[Id]\in K^1(\mathbb{U}_n)$ be the identity element, $f:\tilde{M}\rtimes \Gamma\to U_n^\delta$ a trivialised unitary flat bundle. The pullback of $\Psi_{U_n}([Id])$ by $f$ which is an element in $\KK^1_{C(M),\R}(C(M),C(M))=K^1(M)\otimes \R$ is equal to the $\alpha$ invariant of $f$.

In other words $\Psi_{U_n}([Id])$ is a classifying element for Chern-Simons invariants in $\KK$-theory.
\end{theorem}
\begin{proof}
We will first describe $\Psi_{U_n}([Id])$. We will check that the final element obtained is the map $T$ constructed in \ref{prop chern simons M T V times W}. The result then follows from \Cref{main prop}. The following enumeration follows each successive composition starting with step $0$ to denote the element $[Id]\in K^1(\mathbb{U}_n)$.
\begin{enumerate}[start=0]
\item The identity element $[Id]\in K^1(\mathbb{U}_n)$ will be seen as a unitary automorphism of total space of the trivial bundle of rank $n$ over $\mathbb{U}_n$ given by $$\mathbb{U}_n\times\C^n\to \mathbb{U}_n\times\C^n,\quad  (x,v)\to (x,xv).$$
\item The first morphism changes this element to become a a unitary isomorphism from $\mathbb{U}_n\times \mathbb{U}_n\times\C^n\to \mathbb{U}_n\times \mathbb{U}_n\times \C^n$ sending $L(x,y,u)= (x,y,yx^{-1}v)$. This will be regarded as the composition of two isomorphisms $L=L_2L^{-1}_1$ \begin{align*}
L_i:\mathbb{U}_n\times \mathbb{U}_n\times \C^n\to \mathbb{U}_n\times \mathbb{U}_n\times\C^n\\ L_i(x_1,x_2,v)=(x_1,x_2,x_iv).
\end{align*}  Notice that the group $U_n$ acts trivially on $\C^n$ and $L$ is $U_n$-equivariant. The group $U_n$ doesn't act trivially on the $\C^n$ appearing in the domain of the maps $L_1$, and $L_2$. It acts by $z\cdot(x_1,x_2,v)=(x_1z^{-1},x_2z^{-1},zv).$ Both the maps $L_1$ and $L_2$ become equivariant for this action. We have \begin{align*}L_i(z\cdot(x_1,x_2,v))=L_i(x_1z^{-1},x_2z^{-1},zv)&=(x_1z^{-1},x_2z^{-1},x_iz^{-1}zv)\\&=(x_1z^{-1},x_2z^{-1},x_iv)\\&=z\cdot L_i(x_1,x_2,v).
\end{align*}
\item This is the forgetful map, only changing the topology in the last picture of the group $U_n$ to $U_n^\delta$.
\item One views the bundles $\mathbb{U}_n\times \mathbb{U}_n\times \C^n$ as a bundle over the first copy of $\mathbb{U}_n$ with coefficients in $C(\mathbb{U}_n)$. Applying \Cref{KK theory inclusion- morita then R KK} amounts to extending the coefficient algebra to $C(\mathbb{U}_n)\rtimes \mathbb{U}_n^\delta$.

\item We will use the notation of \Cref{prop chern simons M T V times W}, and let $P=\tilde{M}\times_{\Gamma}U_n$.

The pull back of the element obtained in step 3 by $\phi$, becomes the vector bundle $\C^n\otimes W$ over $M$. The middle vector bundle in step 1, becomes $V\otimes W$, $L_1^{-1}$ and $L_2$ become respectively $T_2\otimes Id_W$, and $T_1$. Applying the Morita equivalence between $M$ and $\tilde{M}\rtimes \Gamma$ finishes the proof.\qedhere
\end{enumerate}
\end{proof}
\begin{rem}
In most of this article compactness is not needed. Most notably, let $\phi:\pi_1(M)\to U_n$ be the holonomy representation of a trivialised unitary flat vector bundle on a not necessarily compact manifold, and $f_\phi:M\to U_n\rtimes U_n^\delta$ the corresponding generalised morphism. The element $f^*\Phi_{U_n}(Id_n)$ is an element in $\KK^1_{M,\R}(C_0(M),C_0(M))$. This later group is isomorphic to the $K$-theory with real coefficients without compact support as proved in proposition 2.20 \cite{MR918241}.
\end{rem}

\bibliographystyle{plain}
\end{document}